\documentclass{amsart}

\usepackage{amsmath,amssymb,amsthm,latexsym}

\newtheorem{lemma}{Lemma}
\newtheorem{theorem}{Theorem}

\theoremstyle{definition}
\newtheorem*{remark}{Remark}

\begin{document}

\title{Enveloping algebras of Malcev algebras}

\author[Bremner]{Murray R. Bremner}

\address{Department of Mathematics and Statistics, University of Saskatchewan,
Canada}

\email{bremner@math.usask.ca}

\author[Hentzel]{Irvin R. Hentzel}

\address{Department of Mathematics, Iowa State University, USA}

\email{hentzel@iastate.edu}

\author[Peresi]{Luiz A. Peresi}

\address{Department of Mathematics, University of S\~ao Paulo, Brazil}

\email{peresi@ime.usp.br}

\author[Tvalavadze]{Marina V. Tvalavadze}

\address{Department of Mathematics and Statistics, University of Saskatchewan,
Canada}

\email{tvalavadze@math.usask.ca}

\author[Usefi]{Hamid Usefi}

\address{Department of Mathematics, University of British Columbia, Canada}

\email{usefi@math.ubc.ca}

\thanks{This is a revised and expanded version of the survey talk given by the first
author at the Second Mile High Conference on Nonassociative Mathematics at the
University of Denver (Denver, Colorado, USA, June 22 to 26, 2009)}

\begin{abstract}
We first discuss the construction by P\'erez-Izquierdo and Shestakov of universal
nonassociative enveloping algebras of Malcev algebras.  We then describe recent
results on explicit structure constants for the universal enveloping algebras (both
nonassociative and alternative) of the 4-dimensional solvable Malcev algebra and
the 5-dimensional nilpotent Malcev algebra.  We include a proof (due to Shestakov)
that the universal alternative enveloping algebra of the real 7-dimensional simple
Malcev algebra is isomorphic to the 8-dimensional division algebra of real octonions.
We conclude with some brief remarks on tangent algebras of analytic Bol loops and
monoassociative loops.
\end{abstract}

\maketitle

\section{Introduction}

Moufang-Lie algebras were introduced by Malcev \cite{Malcev} as the tangent
algebras of analytic Moufang loops. These structures were given the name
Malcev algebras by Sagle \cite{Sagle}.  Thus Malcev algebras are related to
alternative algebras in the same way that Lie algebras are related to associative
algebras.

In 2004, P\'erez-Izquierdo and Shestakov \cite{PerezIzquierdoShestakov}
extended the famous PBW theorem (Poincar\'e-Birkhoff-Witt) from Lie algebras to
Malcev algebras. For any Malcev algebra $M$, over a field of characteristic
$\ne 2, 3$, they constructed a universal nonassociative enveloping algebra
$U(M)$, which shares many properties of the universal associative enveloping
algebras of Lie algebras: $U(M)$ is linearly isomorphic to the polynomial
algebra $P(M)$ and has a natural (nonassociative) Hopf algebra structure. We
will begin by describing the construction of $U(M)$. We will then show how this
construction can be used to compute explicit structure constants for enveloping
algebras of low-dimensional Malcev algebras. This involves differential
operators on $P(M)$, which is the associated graded algebra of $U(M)$, and
derivations defined by two elements of $N_\mathrm{alt}(M)$, the generalized
alternative nucleus of $U(M)$.

Since in general $U(M)$ is not alternative, it is interesting to calculate its
alternator ideal $I(M)$ and its maximal alternative quotient $A(M) =
U(M)/I(M)$, which is the universal alternative enveloping algebra of $M$. This
produces new examples of infinite dimensional alternative algebras.

This program has been carried out for the 4-dimensional solvable Malcev
algebra and for the 5-dimensional nilpotent Malcev algebra. It is work in
progress for the one-parameter family of 5-dimensional solvable (non-nilpotent)
Malcev algebras, and the 5-dimensional non-solvable Malcev algebra. The last
algebra is especially interesting since it is the split extension of the simple
3-dimensional Lie algebra by its unique irreducible non-Lie module.

The ultimate goal of this research program is to calculate the structure constants
for $U(\mathbb{M})$ where $\mathbb{M}$ is the 7-dimensional simple Malcev
algebra.  It is known that in this case the universal alternative enveloping algebra
is just the octonion algebra; see Section \ref{shestakovproof} for a proof.

\section{Lie algebras and Malcev algebras}

Recall that a Lie algebra is a vector space $L$ over a field $F$ with a
bilinear product $[a,b]$ satisfying anticommutativity and the Jacobi identity
for all $a, b, c \in L$:
  \[
  [a,a] = 0,
  \qquad
  J(a,b,c) = 0,
  \qquad
  J(a,b,c) = [[a,b],c] + [[b,c],a] + [[c,a],b].
  \]
If $A$ is an associative algebra over $F$ then we obtain a Lie algebra $A^-$ by
keeping the same underlying vector space but replacing the associative product
$ab$ by the commutator $[a,b] = ab - ba$.  Conversely, the PBW theorem implies
that every Lie algebra $L$ has a universal associative enveloping algebra
$U(L)$ for which the map $L \to U(L)$ is injective, so that $L$ is isomorphic
to a subalgebra of $U(L)^-$. If $\mathrm{char}\,F = 0$ then the image of $L$ in
$U(L)$ consists exactly of the elements of $U(L)$ which are primitive (in the Hopf
algebra sense) with respect to the coproduct $\Delta\colon U(L) \to U(L) \otimes U(L)$.

A Malcev algebra is a vector space $M$ over a field $F$ with a bilinear product
$[a,b]$ satisfying anticommutativity and the Malcev identity for all $a, b, c
\in M$:
  \[
  [a,a] = 0,
  \qquad
  [J(a,b,c),a] = J(a,b,[a,c]).
  \]
It is clear that every Lie algebra is a Malcev algebra.

Associative algebras are defined by the identity $(a,b,c) = 0$ where $(a,b,c) =
(ab)c - a(bc)$ is the associator.  We can weaken this condition by requiring
only that the associator should be an alternating function of its arguments.
This gives the left and right alternative identities, defining the variety of
alternative algebras: $(a,a,b) = 0$ and $(b,a,a) = 0$. If $A$ is an alternative
algebra over $F$ then we obtain a Malcev algebra $A^-$ by keeping the same
underlying vector space but replacing the alternative product $ab$ by the
commutator $[a,b] = ab - ba$. It is an open problem whether every Malcev
algebra $M$ has a universal alternative enveloping algebra $A(M)$ for which the
map $M \to A(M)$ is injective.  In other words, it is not known whether every
Malcev algebra is special: that is, isomorphic to a subalgebra of $A^-$ for some
alternative algebra $A$.

A solution to a closely related problem was given in 2004 by P\'erez-Izquierdo
and Shestakov \cite{PerezIzquierdoShestakov}: they constructed universal
nonassociative enveloping algebras for Malcev algebras. This seems to be the
natural generalization of the PBW theorem to Malcev algebras, since the
universal nonassociative enveloping algebra $U(M)$ of a Malcev algebra $M$ has
a natural (nonassociative) Hopf algebra structure and an associated graded algebra
which is a (commutative associative) polynomial algebra.

\section{Theorem of P\'erez-Izquierdo and Shestakov}

For a nonassociative algebra $A$ we define the generalized alternative nucleus
by
  \[
  N_{\mathrm{alt}}(A)
  =
  \big\{ \,
  a \in A
  \mid
  (a,b,c) = -(b,a,c) = (b,c,a), \, \forall \, b, c \in A
  \, \big\}.
  \]
In general $N_{\mathrm{alt}}(A)$ is not a subalgebra of $A$ but it is closed
under the commutator (it is a subalgebra of $A^-$) and is a Malcev algebra.

\begin{theorem}
\emph{(P\'erez-Izquierdo and Shestakov \cite{PerezIzquierdoShestakov})} For
every Malcev algebra $M$ over a field $F$ of characteristic $\ne 2, 3$ there
exists a nonassociative algebra $U(M)$ and an injective algebra morphism
$\iota\colon M \to U(M)^-$ such that $\iota(M) \subseteq
N_{\mathrm{alt}}(U(M))$; furthermore, $U(M)$ is a universal object with respect
to such morphisms.
\end{theorem}

In general $U(M)$ is not alternative but it has a monomial basis of PBW type; over
a field of characteristic 0, the elements $\iota(M)$ are exactly the primitive elements
of $U(M)$ with respect to the coproduct $\Delta\colon U(M) \to U(M) \otimes U(M)$.

Let $F(M)$ be the unital free nonassociative algebra on a basis of $M$. Let
$R(M)$ be the ideal of $F(M)$ generated by the elements
  \[
  ab - ba - [a,b],
  \qquad
  (a,x,y) + (x,a,y),
  \qquad
  (x,a,y) + (x,y,a),
  \]
for all $a, b \in M$ and all $x, y \in F(M)$. Define $U(M) = F(M)/R(M)$ with
the natural mapping
  \[
  \iota\colon M \to N_\mathrm{alt}(U(M)) \subseteq U(M),
  \qquad
  a \mapsto \iota(a) = \overline{a} = a + R(M).
  \]
Since $\iota$ is injective, we can identify $M$ with $\iota(M) \subseteq U(M)$.
We fix a basis $B = \{ a_i \,|\, i \in \mathcal{I} \}$ of $M$ and a total order
$<$ on $\mathcal{I}$. Define
  \[
  \Omega =
  \{ \,
  (i_1,\hdots,i_n)
  \mid
  n \ge 0, \,
  i_1 \le \cdots \le i_n \,
  \}.
  \]
For $n = 0$ we have $\overline{a}_\emptyset = 1 \in U(M)$, and for $n \ge 1$
the $n$-tuple $(i_1,\hdots,i_n) \in \Omega$ defines a left-tapped monomial
  \[
  \overline{a}_I =
  \overline{a}_{i_1} (
  \overline{a}_{i_2} ( \cdots
  ( \overline{a}_{i_{n-1}} \overline{a}_{i_n} )
  \cdots )),
  \qquad
  |\overline{a}_I| = n.
  \]
In P\'erez-Izquierdo and Shestakov \cite{PerezIzquierdoShestakov} it is shown
that the set of all $\overline{a}_I$ for $I \in \Omega$ is a basis of $U(M)$.

We write $P(M)$ for the polynomial algebra on the vector space $M$.  It follows
that there is a linear isomorphism
  \[
  \phi\colon U(M) \to P(M),
  \qquad
  \overline{a}_{i_1}
  (
  \overline{a}_{i_2}
  ( \cdots
  (
  \overline{a}_{i_{n-1}} \overline{a}_{i_n}
  )
  \cdots ))
  \mapsto
  a_{i_1}
  a_{i_2}
  \cdots
  a_{i_{n-1}} a_{i_n}.
  \]
We now define linear mappings
  \[
  \rho\colon U(M) \to \mathrm{End}_F P(M),
  \qquad
  \lambda\colon U(M) \to \mathrm{End}_F P(M).
  \]
For $x \in U(M)$ we write $\rho(x)$, respectively $\lambda(x)$, for the linear
operator on $P(M)$ induced by the right bracket $y \mapsto [y,x]$ in $U(M)$,
respectively the left multiplication $y \mapsto xy$ in $U(M)$:
  \[
  \rho(x)(f) = \phi \big( [ \phi^{-1}(f), x ] \big),
  \qquad
  \lambda(x)(f) = \phi \big( x \phi^{-1}(f) \big).
  \]
We use the linear operators $\rho(x)$ and $\lambda(x)$ to express commutation
and multiplication in $U(M)$ in terms of differential operators on the
polynomial algebra $P(M)$.

Since $M \subseteq N_\mathrm{alt}(U(M))$, for any $a, b \in M$ and $x \in U(M)$
we have
  \[
  (a,b,x)
  =
  \tfrac16
  [[x,a],b]
  -
  \tfrac16
  [[x,b],a]
  -
  \tfrac16
  [x,[a,b]].
  \]
From this follow the next three lemmas, which are implicit in P\'erez-Izquierdo and
Shestakov \cite{PerezIzquierdoShestakov}.

\begin{lemma}
Let $x$ be a basis monomial of $U(M)$ with $|x| \ge 2$ and write $x = by$ with
$b \in M$.  Then for any $a \in M$ we have
  \[
  [x,a]
  =
  [by,a]
  =
  [b,a]y
  +
  b[y,a]
  +
  \tfrac12
  [[y,a],b]
  -
  \tfrac12
  [[y,b],a]
  -
  \tfrac12
  [y,[a,b]].
  \]
\end{lemma}

\begin{lemma}
Let $x$ be a basis monomial of $U(M)$ with $|x| \ge 2$ and write $x = by$ with
$b \in M$.    Then for any $a \in M$ we have
  \[
  ax
  =
  a(by)
  =
  b(ay)
  +
  [a,b]y
  -
  \tfrac13
  [[y,a],b]
  +
  \tfrac13
  [[y,b],a]
  +
  \tfrac13
  [y,[a,b]].
  \]
\end{lemma}

\begin{lemma}
Let $y$ and $z$ be basis monomials of $U(M)$ with $|y| \ge 2$. Write $y = ax$
with $a \in M$. Then
  \[
  yz = (ax)z = 2 a(xz) - x(az) - x[z,a] + [xz,a].
  \]
\end{lemma}

We can use a result of Morandi, P\'erez-Izquierdo and Pumpl\"un
\cite{MorandiPerezIzquierdoPumplun} to express these lemmas in terms of
derivations.  Let $A$ be a nonassociative algebra, and let $a, b \in
N_\mathrm{alt}(A)$ be any two elements of the generalized alternative nucleus.
As usual we define the following operators on $A$:
  \[
  L_a(x) = ax,
  \qquad
  R_a(x) = xa,
  \qquad
  \mathrm{ad}_a(x) = [a,x].
  \]
Then the following operator is a derivation of $A$:
  \[
  D_{a,b}
  =
  [ L_a, L_b ] + [ L_a, R_b ] + [ R_a, R_b ]
  =
  \mathrm{ad}_{[a,b]} - 3 [ L_a, R_b ]
  =
  \tfrac12 \mathrm{ad}_{[a,b]} + \tfrac12 [ \mathrm{ad}_{a},\mathrm{ad}_{b} ].
  \]
To apply this to the universal enveloping algebra of a Malcev algebra $M$, we
take $A = U(M)$ and $a, b \in M \subseteq N_\mathrm{alt}(A)$. It is clear that
$D_{a,a} = 0$ for all $a$, and that $D_{b,a} = - D_{a,b}$ for all $a, b$. For
any $a, b \in M$ and any $x \in U(M)$ we have
  \[
  6 (b,x,a)
  =
  - 2 D_{a,b}(x) - 2 [x,[a,b]].
  \]
Using this, we can rewrite the first two lemmas as follows:
  \allowdisplaybreaks
  \begin{align*}
  [x,a]
  &=
  [b,a]y + b[y,a] - D_{a,b}(y) - [y,[a,b]],
  \\
  ax
  &=
  b(ay) + [a,b]y + \tfrac23 D_{a,b}(y) + \tfrac23 [y,[a,b]].
  \end{align*}
These equations allow us to inductively build up the structure constants for
the universal nonassociative enveloping algebra $U(M)$.

\section{The 4-dimensional Malcev algebra}

The first application of the work of P\'erez-Izquierdo and Shestakov to the
computation of explicit structure constants for $U(M)$ was done by Bremner,
Hentzel, Peresi and Usefi \cite{BremnerHentzelPeresiUsefi} in the case of the
4-dimensional Malcev algebra.  We write $\mathbb{S}$ for this algebra; it is
the smallest non-Lie Malcev algebra, and it is solvable.  It has basis $\{ a,
b, c, d\}$ with structure constants in Table \ref{table1}.

  \begin{table}[ht]
  \begin{center}
  \begin{tabular}{r|rrrr}
  $[{-},{-}]$ & $ a$ & $  b$ & $ c$ & $\phantom{-}d$ \\
  \hline
  &&&& \\[-10pt]
  $a$   & $ 0$ & $ -b$ & $-c$ & $d$ \\
  $b$   & $ b$ & $  0$ & $2d$ & $0$ \\
  $c$   & $ c$ & $-2d$ & $ 0$ & $0$ \\
  $d$   & $-d$ & $  0$ & $ 0$ & $0$
  \end{tabular}
  \end{center}
  \caption{The 4-dimensional Malcev algebra $\mathbb{S}$}
  \label{table1}
  \end{table}

In $\mathbb{S}$ the sets $\{ a, b \}$, $\{ a, c \}$, $\{ a, d \}$ span
2-dimensional solvable Lie subalgebras, and the set $\{ b, c, d \}$ spans a
3-dimensional nilpotent Lie subalgebra.  Therefore these sets generate
associative subalgebras of $U(\mathbb{S})$ with the following structure
constants.

\begin{lemma}
For $e \in \{ b, c \}$ these equations hold in $U(\mathbb{S})$:
  \[
  ( a^i e^j ) ( a^k e^\ell )
  =
  a^i ( a {+} j )^k e^{j+\ell},
  \quad
  ( a^i d^j ) ( a^k d^\ell )
  =
  a^i ( a {-} j )^k d^{j+\ell}.
  \]
\end{lemma}

\begin{lemma}
This equation holds in $U(\mathbb{S})$:
  \[
  ( b^i c^j d^k ) ( b^\ell c^m d^n )
  =
  \sum_{h=0}^\ell
  (-1)^h 2^h \binom{\ell}{h} \frac{j!}{(j{-}h)!}
  b^{i+\ell-h} c^{j+m-h} d^{k+n+h}.
  \]
\end{lemma}

We define the following operators on the polynomial algebra $P(\mathbb{S})$:
  \begin{itemize}
  \item
$I$ is the identity;
  \item
$M_x$ is multiplication by $x \in \{a,b,c,d\}$;
  \item
$D_x$ is differentiation with respect to $x \in \{a,b,c,d\}$;
  \item
$S$ is the shift $a \mapsto a{+}1$: $S( a^i b^j c^k d^\ell ) = (a{+}1)^i b^j
c^k d^\ell$.
  \end{itemize}

\begin{lemma}
For $x, y \in \{a,b,c,d\}$ we have
  \allowdisplaybreaks
  \begin{alignat*}{3}
  &[ D_x, M_x ] = I,
  &\quad
  &[ D_x, M_y ] = 0 \; (x \ne y),
  &\quad
  &[ D_x, D_y ] = [ M_x, M_y ] = 0,
  \\
  &[ M_a, S ] = - S,
  &\quad
  &[ M_x, S ] = 0 \; (x \ne a),
  &\quad
  &[ D_x, S ] = [D_x,S^{-1}] = 0,
  \\
  &[M_a,S^{-1}] = S^{-1},
  &\quad
  &[M_x,S^{-1}] = 0 \; (x \ne a).
  \end{alignat*}
\end{lemma}

We can now determine $\rho(x)$ and $\lambda(x)$ for $x \in \{ a, b, c, d \}$ as
differential operators on the polynomial algebra $P(\mathbb{S})$; see Table
\ref{tableS}. The following lemma gives an example of our inductive proof techniques.

  \begin{table}[ht]
  \begin{center}
  \begin{tabular}{l|l|l}
  $x$
  &
  $\rho(x)$
  &
  $\lambda(x)$
  \\ \hline
  && \\[-10pt]
  $a$
  &
  $M_b D_b {+} M_c D_c {-} M_d D_d {-} 3 M_d D_b D_c$
  &
  $M_a$
  \\
  $b$
  &
  $( I {-} S ) M_b + \left( S {-} I {-} 2S^{-1} \right) M_d D_c$
  &
  $S M_b + \left( S^{-1} {-} S \right) M_d D_c$
  \\
  $c$
  &
  $( I {-} S ) M_c + \left( S {-} I {+} 2 S^{-1} \right) M_d D_b$
  &
  $S M_c - \left( S^{-1} {+} S \right) M_d D_b$
  \\
  $d$
  &
  $\left( I {-} S^{-1} \right) M_d$
  &
  $S^{-1} M_d$
  \end{tabular}
  \end{center}
  \caption{Differential operators $\lambda(x)$ and $\rho(x)$ for $\mathbb{S}$}
  \label{tableS}
  \end{table}

\begin{lemma}\label{bcdabracket}
We have
  $[ b^n c^p d^q, a ]
  =
  (n {+} p {-} q) b^n c^p d^q - 3np b^{n-1} c^{p-1} d^{q+1}$.
\end{lemma}

\begin{proof}
By induction on $n$.  The basis $n = 0$ is $[ c^p d^q, a ] = (p {-} q) c^p d^q$,
since $a,c,d$ span a Lie subalgebra of $\mathbb{M}$. Now let $n \ge 0$ and use
the right-bracket lemma:
  \allowdisplaybreaks
  \begin{align*}
  &
  [ b^{n+1} c^p d^q, a ]
  =
  \\
  &
  [ b, a ] b^n c^p d^q
  +
  b [ b^n c^p d^q, a ]
  +
  \tfrac12 \big(
  [ [ b^n c^p d^q, a ], b ]
  -
  [ [ b^n c^p d^q, b ], a ]
  -
  [ b^n c^p d^q, [ a, b ] ]
  \big).
  \end{align*}
We use $[a,b] = -b$ and then apply the structure constants for the nilpotent
Lie subalgebra spanned by $b,c,d$:
  \allowdisplaybreaks
  \begin{align*}
  &
  [ b^{n+1} c^p d^q, a ]
  =
  \\
  &
  b^{n+1} c^p d^q
  +
  b [ b^n c^p d^q, a ]
  +
  \tfrac12
  [ [ b^n c^p d^q, a ], b ]
  +
  p [ b^n c^{p-1} d^{q+1}, a ]
  -
  p b^n c^{p-1} d^{q+1}.
  \end{align*}
We now apply the inductive hypothesis to obtain
  \allowdisplaybreaks
  \begin{align*}
  &[ b^{n+1} c^p d^q, a ]
  =
  b^{n+1} c^p d^q
  +
  (n{+}p{-}q) b^{n+1} c^p d^q - 3np b^n c^{p-1} d^{q+1}
  \\
  &+
  \tfrac12
  (n{+}p{-}q) [ b^n c^p d^q, b ]
  -
  \tfrac32
  np [ b^{n-1} c^{p-1} d^{q+1}, b ]
  +
  (n{+}p{-}q{-}2) p
  b^n c^{p-1} d^{q+1}
  \\
  &
  -
  3 n p (p{-}1)
  b^{n-1} c^{p-2} d^{q+2}
  -
  p
  b^n c^{p-1} d^{q+1}.
  \end{align*}
We use the nilpotent Lie subalgebra again to get
  \allowdisplaybreaks
  \begin{align*}
  &[ b^{n+1} c^p d^q, a ]
  =
  b^{n+1} c^p d^q
  +
  (n{+}p{-}q) b^{n+1} c^p d^q
  -
  3np \, b^n c^{p-1} d^{q+1}
  \\
  &
  -
  (n{+}p{-}q) p b^n c^{p-1} d^{q+1}
  +
  3
  n p (p{-}1) b^{n-1} c^{p-2} d^{q+2}
  +
  (n{+}p{-}q{-}2) p b^n c^{p-1} d^{q+1}
  \\
  &
  -
  3 n p (p{-}1)
  b^{n-1} c^{p-2} d^{q+2}
  -
  p
  b^n c^{p-1} d^{q+1}.
  \end{align*}
Combining terms now gives
  \[
  [ b^{n+1} c^p d^q, a ]
  =
  (n{+}1{+}p{-}q) b^{n+1} c^p d^q - 3 (n{+}1) p b^n c^{p-1} d^{q+1},
  \]
and this completes the proof.
\end{proof}

It is not hard to show that $[ a^m b^n c^p d^q, a ] = a^m [ b^n c^p d^q, a ]$;
combining this with the last result gives the formula for $\rho(a)$ in Table
\ref{tableS}.

Our next task is to determine $\lambda(x)$ for $x = a^i b^j c^k d^\ell$ as a
differential operator on $P(\mathbb{S})$. The proofs use the fact that if linear
operators $E$ and $F$ satisfy $[[E,F],F] = 0$ then $[ E, F^k ] = k [E,F]
F^{k-1}$ for $k \ge 1$. Since $c, d$ span an abelian Lie subalgebra
$\mathbb{A} \subset \mathbb{S}$, associativity gives
$\lambda(c^k d^\ell) = \lambda(c)^k \lambda(d)^\ell$
on $U(\mathbb{A})$; this also holds on $U(\mathbb{S})$.

\begin{lemma}
In $U(\mathbb{S})$ we have $\lambda( c^k d^\ell ) = \lambda(c)^k
\lambda(d)^\ell$.
\end{lemma}

Since $b, c, d$ span a nilpotent Lie subalgebra $\mathbb{N} \subset
\mathbb{S}$, associativity gives $\lambda(b^j c^k d^\ell)$ $=$ $\lambda(b)^j
\lambda(c)^k \lambda(d)^\ell$ on $U(\mathbb{N})$; this fails on
$U(\mathbb{S})$.

\begin{lemma}
In $U(\mathbb{S})$ the operator $\lambda( b^j c^k d^\ell )$ equals
  \[
  \sum_{\alpha=0}^{\min(j,k)}
  \sum_{\beta=0}^\alpha
  (-1)^{\alpha-\beta}
  \alpha!
  \binom{\alpha}{\beta}
  \binom{j}{\alpha}
  \binom{k}{\alpha}
  S^{-\beta}
  \lambda(b)^{j-\alpha}
  \lambda(c)^{k-\alpha}
  M_d^\alpha
  \lambda(d)^\ell.
  \]
\end{lemma}

The inductive proof of this result is similar to, but much more complicated
than, the example given above.

Multiplication by the general basis monomial $a^i b^j c^k d^\ell$ is given by
the following formula.

\begin{lemma}
In $U(\mathbb{S})$ the operator $\lambda( a^i b^j c^k d^\ell )$ equals
  \allowdisplaybreaks
  \begin{align*}
  &
  \sum_{\alpha=0}^{\min(j,k)}
  \sum_{\beta=0}^{\alpha}
  \sum_{\gamma=0}^{i}
  \sum_{\delta=0}^{i-\gamma}
  \sum_{\epsilon=0}^{i-\gamma-\delta}
  (-1)^{i+\alpha-\beta-\gamma-\delta}
  \alpha! \delta! \epsilon!
  \binom{\alpha}{\beta}
  \binom{j}{\alpha,\epsilon}
  \binom{k}{\alpha,\delta}
  \times
  \\
  &
  X_i(\gamma,\delta,\epsilon)
  \lambda(a)^\gamma
  S^{-\beta-\delta-\epsilon}
  \lambda(b)^{j-\alpha-\epsilon}
  D_b^\delta
  D_c^\epsilon
  \lambda(c)^{k-\alpha-\delta}
  M_d^{\alpha+\delta+\epsilon}
  \lambda(d)^\ell,
  \end{align*}
where $X_i(\gamma,\delta,\epsilon)$ is a polynomial in $\alpha{-}\beta$
satisfying the recurrence
  \allowdisplaybreaks
  \begin{align*}
  X_{i+1}(\gamma,\delta,\epsilon)
  &=
  (\alpha{-}\beta{+}\delta{+}\epsilon)
  X_i(\gamma,\delta,\epsilon)
  \\
  &\quad
  +
  X_i(\gamma{-}1,\delta,\epsilon)
  +
  X_i(\gamma,\delta{-}1,\epsilon)
  +
  X_i(\gamma,\delta,\epsilon{-}1),
  \end{align*}
with the initial conditions $X_0(0,0,0) = 1$ and $X_i(\gamma,\delta,\epsilon) = 0$
unless $0 \le \gamma \le i$, $0 \le \delta \le i {-} \gamma$, $0 \le \epsilon \le i {-}
\gamma {-} \delta$.
\end{lemma}

The unique solution to the recurrence in this lemma is
  \[
  X_i(\gamma,\delta,\epsilon)
  =
  \binom{\delta{+}\epsilon}{\epsilon}
  \sum_{\zeta=0}^{i{-}\gamma{-}\delta{-}\epsilon}
  \binom{i}{\gamma,\zeta}
  \left\{ \begin{matrix}
  i{-}\gamma{-}\zeta \\ \delta{+}\epsilon
  \end{matrix} \right\}
  (\alpha{-}\beta)^\zeta,
  \]
where the Stirling numbers of the second kind are defined by
  \[
  \left\{ \begin{matrix} r \\ s \end{matrix} \right\}
  =
  \frac{1}{s!}
  \sum_{t=0}^s
  (-1)^{s-t}
  \binom{s}{t}
  t^r.
  \]
We are almost ready to write down the structure constants for the universal
nonassociative enveloping algebra $U(\mathbb{S})$.  We need the differential
coefficients defined by
  \[
  \left[ \begin{matrix} r \\ 0 \end{matrix} \right]
  =
  1,
  \;
  \left[ \begin{matrix} r \\ s \end{matrix} \right]
  =
  r(r{-}1)\cdots(r{-}s{+}1),
  \;
  \text{so that}
  \;
  D_x^s ( x^r )
  =
  \left[ \begin{matrix} r \\ s \end{matrix} \right]
  x^{r-s}.
  \]

\begin{theorem}
\emph{(Bremner, Hentzel, Peresi and Usefi \cite{BremnerHentzelPeresiUsefi})} In
$U(\mathbb{S})$, the universal nonassociative enveloping algebra of the
4-dimensional Malcev algebra $\mathbb{S}$, the product of basis monomials
$( a^i b^j c^k d^\ell ) ( a^m b^n c^p d^q )$ equals
  \allowdisplaybreaks
  \begin{align*}
  &
  \sum_{\alpha=0}^{\min(j,k)}
  \sum_{\beta=0}^{\alpha}
  \sum_{\gamma=0}^{i}
  \sum_{\delta=0}^{i-\gamma}
  \sum_{\epsilon=0}^{i-\gamma-\delta}
  \sum_{\zeta=0}^{i{-}\gamma{-}\delta{-}\epsilon}
  \sum_{\eta=0}^{j-\alpha-\epsilon}
  \sum_{\theta=0}^{j-\alpha-\epsilon-\eta}
  \sum_{\lambda=0}^{k-\alpha-\delta}
  \sum_{\mu=0}^{k-\alpha-\delta-\lambda}
  \sum_{\nu=0}^m
  \\
  &
  (-1)^{i+j+k+\alpha-\beta-\gamma-\epsilon-\eta-\theta-\lambda}
  (\alpha{-}\beta)^\zeta
  \alpha!
  \binom{\alpha}{\beta}
  (\delta{+}\epsilon)!
  \omega^\nu
  \binom{i}{\gamma,\zeta}
  \times
  \\
  &
  \left\{ \begin{matrix}
  i{-}\gamma{-}\zeta \\ \delta{+}\epsilon
  \end{matrix} \right\}
  \binom{j}{\alpha,\epsilon,\eta,\theta}
  \binom{k}{\alpha,\delta,\lambda,\mu}
  \binom{m}{\nu}
  \left[
  \begin{matrix}
  n \\
  k{-}\alpha{-}\lambda
  \end{matrix}
  \right]
  \left[
  \begin{matrix}
  p{+}\lambda \\
  j{-}\alpha{-}\eta
  \end{matrix}
  \right]
  \times
  \\
  &
  a^{m+\gamma-\nu}
  b^{-k+n+\alpha+\eta+\lambda}
  c^{-j+p+\alpha+\eta+\lambda}
  d^{j+k+\ell+q-\alpha-\eta-\lambda},
  \end{align*}
where $\omega = j {+} k {-} \ell {-} 2\alpha {-} \beta {-} 2\delta {-}
2\epsilon {-} 2\theta {-} 2\mu$. (We make the convention that when $\alpha =
\beta$ and $\zeta = 0$ we set $(\alpha{-}\beta)^\zeta = 1$.)
\end{theorem}

It is not difficult to show that $U(\mathbb{S})$ is not alternative; in fact,
it is not even power-associative.  So it is interesting to determine the
largest alternative quotient of $U(\mathbb{S})$. The alternator ideal in a
nonassociative algebra is the ideal generated by all elements of the form
$(x,x,y)$ and $(y,x,x)$. Let $M$ be a Malcev algebra, $U(M)$ its universal
enveloping algebra, and $I(M) \subseteq U(M)$ the alternator ideal. The
universal alternative enveloping algebra of $M$ is $A(M) = U(M)/I(M)$.

\begin{lemma}
We have the following nonzero alternators in $U(\mathbb{S})$:
  \[
  ( a, bc, bc ) = 2 d^2,
  \qquad
  ( b, ac, ac ) = cd,
  \qquad
  ( c, ab, ab ) = - bd.
  \]
\end{lemma}

\begin{proof}
From the structure constants for $U(\mathbb{S})$ we obtain
  \[
  \big( a (bc) \big) (bc) = a b^2 c^2 - 2 abcd + 2 d^2,
  \quad
  a \big( (bc) (bc) \big) = a b^2 c^2 - 2 abcd,
  \]
which imply the first equation. The others are similar.
\end{proof}

Let $J \subseteq U(\mathbb{S})$ be the ideal generated by $\{ d^2, cd, bd \}$.
In $U(\mathbb{S})/J$ it suffices to consider two types of monomials: $a^i d$
and $a^i b^j c^k$. If $m$ is one of these, we write $m$ when we mean $m+J$ in
the next theorem.

\begin{theorem}
\emph{(Bremner, Hentzel, Peresi and Usefi \cite{BremnerHentzelPeresiUsefi})}
In $U(\mathbb{S})/J$ we have
  \allowdisplaybreaks
  \begin{align*}
  ( a^i d ) ( a^m d )
  &=
  0,
  \\
  ( a^i b^j c^k ) ( a^m d )
  &=
  \delta_{j0} \delta_{k0} \, a^{i+m} d,
  \\
  ( a^i d ) ( a^m b^n c^p )
  &=
  \delta_{n0} \delta_{p0} \, a^i (a{-}1)^m d,
  \\
  ( a^i b^j c^k ) ( a^m b^n c^p )
  &=
  a^i (a{+}j{+}k)^m b^{j+n} c^{k+p}
  +
  \delta_{j+n,1} \delta_{k+p,1}
  T^{im}_{jk},
  \end{align*}
where
  \[
  T^{im}_{jk}
  =
  \begin{cases}
  0
  &\text{if $(j,k) = (0,0)$},
  \\
  (a{-}1)^{i+m} d - a^i (a{+}1)^m d
  &\text{if $(j,k) = (1,0)$},
  \\
  - (a{-}1)^{i+m} d - a^i (a{+}1)^m d
  &\text{if $(j,k) = (0,1)$},
  \\
  a^i (a{-}1)^m d - a^i (a{+}2)^m d
  &\text{if $(j,k) = (1,1)$}.
  \end{cases}
  \]
\end{theorem}

It can be shown by direct calculation that the associator in $U(\mathbb{S})/J$
is an alternating function of its arguments. It follows that $J$ equals the
alternator ideal $I(\mathbb{S})$, and that $U(\mathbb{S})/J$ is isomorphic to
the universal alternative enveloping algebra $A(\mathbb{S})$. The first proof
of this result was given by Shestakov \cite{Shestakov1, Shestakov2} using different
methods. Since $I(\mathbb{S})$ contains no elements of degree 1, the natural
mapping from $\mathbb{S}$ to $A(\mathbb{S})$ is injective, and hence
$\mathbb{S}$ is special. This also follows directly from the isomorphism
$\mathbb{S} \cong \mathbb{A}^-$ where $\mathbb{A}$ is the 4-dimensional
alternative algebra in Table \ref{table2}.

  \begin{table}[ht]
  \begin{center}
  \begin{tabular}{r|rrrr}
  $\cdot$ & $\phantom{-}a$ & $\phantom{-}b$ & $\phantom{-}c$ & $\phantom{-}d$ \\
  \hline
  &&&& \\[-10pt]
  $a$   & $ a$ & $ 0$ & $ 0$ & $ d$ \\
  $b$   & $ b$ & $ 0$ & $ d$ & $ 0$ \\
  $c$   & $ c$ & $-d$ & $ 0$ & $ 0$ \\
  $d$   & $ 0$ & $ 0$ & $ 0$ & $ 0$
  \end{tabular}
  \end{center}
  \caption{The 4-dimensional alternative algebra $\mathbb{A}$ for which
  $\mathbb{A}^- = \mathbb{S}$}
  \label{table2}
  \end{table}

\section{Malcev algebras of dimensions 5, 6 and 7}

The 5-dimensional Malcev algebras were classified by Kuzmin \cite{Kuzmin2}:
there is one nilpotent algebra, a one-parameter family of solvable
(non-nilpotent) algebras, and one non-solvable algebra.  A recent preprint by
Gavrilov \cite{Gavrilov} gives a new approach to the classification of the
solvable algebras. The structure constants of the nilpotent Malcev algebra
$\mathbb{T}$ are given in Table \ref{table3}.

\begin{theorem}
\emph{(Bremner and Usefi \cite{BremnerUsefi})}
In $U(\mathbb{T})$, the universal nonassociative enveloping algebra of the
5-dimensional nilpotent Malcev algebra $\mathbb{T}$, the product of basis
monomials $( a^i b^j c^k d^\ell e^m ) ( a^p b^q c^r d^s e^t )$ equals
  \allowdisplaybreaks
  \begin{align*}
  &
  \sum_{\alpha=0}^\ell
  \sum_{\beta=0}^i
  \sum_{\gamma=0}^\beta
  \sum_{\delta=0}^\gamma
  \sum_{\epsilon=0}^{j-\alpha-\delta}
  \sum_{\zeta=0}^{j-\alpha-\delta-\epsilon}
  \sum_{\eta=0}^{\ell-\alpha-(\gamma-\delta)}
  \sum_{\theta=0}^{\ell-\alpha-(\gamma-\delta)-\eta}
  \sum_{\lambda=0}^\eta
  \\
  &
  (-1)^{\beta+\zeta+\ell-\alpha-\gamma-\eta}
  \frac{\alpha!\beta!\lambda!}{2^{\alpha+\gamma}3^{j-\epsilon-\zeta+\ell-\alpha-\eta-\theta}}
  \,\times
  \\
  &
  \binom{\alpha}{\beta{-}\gamma}
  \binom{i}{\beta}
  \binom{j}{\alpha,\delta,\epsilon,\zeta}
  \binom{j{-}\alpha{-}\epsilon{-}\zeta}{\lambda}
  \binom{\ell}{\alpha,\gamma{-}\delta,\eta,\theta}
  \binom{\eta}{\lambda}
  \,\times
  \\
  &
  \begin{bmatrix}
  p \\ j{-}\beta{-}\epsilon{+}\ell{-}\alpha{-}\eta{-}\theta
  \end{bmatrix}
  \begin{bmatrix}
  q \\ \ell{-}\alpha{-}\eta{-}\theta
  \end{bmatrix}
  \begin{bmatrix}
  r \\ \theta
  \end{bmatrix}
  \begin{bmatrix}
  s \\ j{-}\alpha{-}\epsilon{-}\zeta{-}\lambda
  \end{bmatrix}
  \,\times
  \\[6pt]
  &
  a^{i{+}p{-}j{+}\epsilon{-}\ell{+}\alpha{+}\eta{+}\theta}
  b^{\epsilon{+}q{-}\ell{+}\alpha{+}\eta{+}\theta}
  c^{\zeta{+}k{+}r{-}\theta}
  d^{\eta{+}s{+}\alpha{+}\epsilon{+}\zeta{-}j}
  e^{j{-}\alpha{-}\epsilon{-}\zeta{+}\ell{-}\eta{+}m{+}t}.
  \end{align*}
\end{theorem}

  \begin{table}[ht]
  \begin{center}
  \begin{tabular}{r|rrrrr}
  $[-{,}-]$ &\phantom{-} $\phantom{-}a$ & $\phantom{-}b$
  & $\phantom{-}c$ & $\phantom{-}d$ & $\phantom{-}e$
  \\
  \hline
  &&&&& \\[-10pt]
  $a$   & $ 0$ & $ c$ & $ 0$ & $ 0$ & $ 0$ \\
  $b$   & $-c$ & $ 0$ & $ 0$ & $ 0$ & $ 0$ \\
  $c$   & $ 0$ & $ 0$ & $ 0$ & $ e$ & $ 0$ \\
  $d$   & $ 0$ & $ 0$ & $-e$ & $ 0$ & $ 0$ \\
  $e$   & $ 0$ & $ 0$ & $ 0$ & $ 0$ & $ 0$
  \end{tabular}
  \end{center}
  \caption{The 5-dimensional nilpotent Malcev algebra $\mathbb{T}$}
  \label{table3}
  \end{table}

\begin{lemma}
We have the following nonzero alternators in $U(\mathbb{T})$:
  \[
  (ab,ab,d) = -\tfrac16 ce,
  \qquad
  (bd,bd,a^2) = \tfrac{1}{18} e^2.
  \]
\end{lemma}

A basis of the ideal $J$ generated by $\{ ce, e^2 \}$ consists of the monomials
  \[
  \{ a^i b^j c^k d^\ell e^m \,|\, m \ge 2 \}
  \cup
  \{ a^i b^j c^k d^\ell e \,|\, k \ge 1 \}.
  \]

\begin{theorem}
\emph{(Bremner and Usefi \cite{BremnerUsefi})}
The algebra $U(\mathbb{T})/J$ is alternative; we have
  \allowdisplaybreaks
  \begin{align*}
  &
  ( a^i b^j d^\ell e ) ( a^p b^q d^s e )
  =
  0,
  \\
  &
  ( a^i b^j c^k d^\ell ) ( a^p b^q d^s e )
  =
  \delta_{k0}
  a^{i+p} b^{j+q} d^{\ell+s} e,
  \\
  &
  ( a^i b^j d^\ell e ) ( a^p b^q c^r d^s )
  =
  \delta_{r0}
  a^{i+p} b^{j+q} d^{\ell+s} e,
  \\
  &
  ( a^i b^j c^k d^\ell ) ( a^p b^q c^r d^s )
  =
  \sum_{\mu=0}^j
  (-1)^\mu
  \mu!
  \binom{j}{\mu}
  \binom{p}{\mu}
  a^{i+p-\mu}
  b^{j+q-\mu}
  c^{k+r+\mu}
  d^{\ell+s}
  \\
  &\quad
  +
  \delta_{k0} \delta_{r0}
  \left(
    \tfrac16 i j s
  {-} \tfrac16 i \ell q
  {+} \tfrac12 j \ell p
  {+} \tfrac13 j p s
  {-} \tfrac13 \ell p q
  \right)
  a^{i+p-1} b^{j+q-1} d^{\ell+s-1} e
  \\
  &\quad
  -
  \delta_{k0} \delta_{r1} \ell a^{i+p} b^{j+q} d^{\ell+s-1} e.
  \end{align*}
\end{theorem}

The remaining 5-dimensional algebras are an infinite family of solvable
non-nilpotent algebras and a non-solvable algebra.  Bremner and Tvalavadze
are studying the universal enveloping algebras for these algebras.  Before discussing
the non-solvable algebra, we will describe the 7-dimensional simple Malcev algebra,
which is closely related to the 8-dimensional alternative algebra of octonions.

The 8-dimensional alternative division algebra $\mathbb{O}$ of real octonions
has basis $1$, $i$, $j$, $k$, $l$, $m$, $n$, $p$ with structure constants obtained
from the Cayley-Dickson doubling process. We replace the product in $\mathbb{O}$
by the commutator $[x,y] = xy - yx$ and denote the resulting Malcev algebra by
$\mathbb{O}^-$. We have the direct sum of ideals $\mathbb{O}^- = \mathbb{R}
\oplus \mathbb{M}$, where $\mathbb{R}$ denotes the span of 1, and $\mathbb{M}$
denotes the span of the other basis elements. The subalgebra $\mathbb{M}$ is a
simple (non-Lie) Malcev algebra with structure constants in Table \ref{table4}.
We complexify $\mathbb{M}$ and introduce a new basis where $\epsilon^2 = -1$:
  \allowdisplaybreaks
  \begin{alignat*}{4}
  h &= \epsilon l,
  &\quad
  x &= \tfrac12( i - \epsilon m ),
  &\quad
  y &= \tfrac12 ( k - \epsilon p ),
  &\quad
  z &= \tfrac12 ( j - \epsilon n ),
  \\
  &&
  x' &= -\tfrac12 ( i + \epsilon m ),
  &\quad
  y' &= -\tfrac12 ( k + \epsilon p ),
  &\quad
  z' &= -\tfrac12 ( j + \epsilon n ).
  \end{alignat*}
For this new basis we have structure constants in Table \ref{table5}.

  \begin{table}[ht]
  \[
  \begin{array}{c|rrrrrrr}
  [-,-] & i & j & k & l & m & n & p \\ \hline
  &&&&&&& \\[-10pt]
  i &   0 & -2k &  2j &  2m & -2l & -2p &  2n \\
  j & -2k &   0 &  2i &  2n &  2p & -2l & -2m \\
  k &  2j & -2i &   0 &  2p & -2n &  2m & -2l \\
  l & -2m & -2n & -2p &   0 &  2i &  2j &  2k \\
  m &  2l & -2p &  2n & -2i &   0 & -2k &  2j \\
  n &  2p &  2l & -2m & -2j &  2k &   0 & -2i \\
  p & -2n &  2m &  2l & -2k & -2j &  2i &   0
  \end{array}
  \]
  \caption{The 7-dimensional simple Malcev algebra $\mathbb{M}$ (non-split form)}
  \label{table4}
  \[
  \begin{array}{c|rrrrrrr}
  [-,-] & h & x & y & z & x' & y' & z' \\ \hline
  &&&&&&& \\[-10pt]
  h  &   0  &  2x  & 2y   &  2z  & -2x' & -2y' & -2z' \\
  x  & -2x  &   0  & 2z'  & -2y' &   h  &   0  &   0  \\
  y  & -2y  & -2z' &  0   &  2x' &   0  &   h  &   0  \\
  z  & -2z  &  2y' & -2x' &   0  &   0  &   0  &   h  \\
  x' &  2x' &  -h  &  0   &   0  &   0  & -2z  &  2y  \\
  y' &  2y' &   0  & -h   &   0  &  2z  &   0  & -2x  \\
  z' &  2z' &   0  &  0   &  -h  & -2y  &  2x  &   0
  \end{array}
  \]
  \caption{The 7-dimensional simple Malcev algebra $\mathbb{M}$ (split form)}
  \label{table5}
  \end{table}

We observe that the subalgebra $\mathbb{L}$ of $\mathbb{M}$ with basis
$\{ h, x, x' \}$ is isomorphic to the simple Lie algebra $\mathfrak{sl}_2(\mathbb{C})$:
$[ h, x ] = 2x$, $[ h, x' ] = -2x'$, $[ x, x' ] = h$. The subspace $\mathbb{V}$
(resp.~$\mathbb{W}$) with basis $\{ y, z' \}$ (resp.~$\{ z, y' \}$) has trivial product
$[ y, z' ] = 0$ (resp.~$[ z, y' ] = 0$) and is an $\mathbb{L}$-module:
  \allowdisplaybreaks
  \begin{alignat*}{4}
  [ h, y ] &= 2y,
  &\qquad
  [ h, z' ] &= -2z',
  &\qquad
  [ h, z ] &= 2z,
  &\qquad
  [ h, y' ] &= -2y',
  \\
  [ x, y ] &= 2z',
  &\qquad
  [ x, z' ] &= 0,
  &\qquad
  [ x, z ] &= -2y',
  &\qquad
  [ x, y' ] &= 0,
  \\
  [ x', y ] &= 0,
  &\qquad
  [ x', z' ] &= 2y,
  &\qquad
  [ x', z ] &= 0,
  &\qquad
  [ x', y' ] &= -2z.
  \end{alignat*}
We have $\mathbb{M} = \mathbb{L} \oplus \mathbb{V} \oplus \mathbb{W}$: the
direct sum of three $\mathbb{L}$-modules. The $\mathbb{L}$-modules $\mathbb{V}$
and $\mathbb{W}$ are isomorphic by $y \leftrightarrow -z$, $z' \leftrightarrow
y'$; these are 2-dimensional irreducible representations of $\mathbb{L}$ which
are not isomorphic to the natural representation of
$\mathfrak{sl}_2(\mathbb{C})$. This is possible because $\mathbb{V}$ and
$\mathbb{W}$ are Malcev modules which are not Lie modules. (Carlsson
\cite{Carlsson} has shown that this is the only case in which a simple Lie
algebra has a Malcev module which is not a Lie module. Elduque and Shestakov
\cite{ElduqueShestakov} determined all irreducible non-Lie modules for Malcev
algebras without any restriction on the dimension.)  The subalgebras $\mathbb{L}
\oplus \mathbb{V}$ and $\mathbb{L} \oplus \mathbb{W}$ are isomorphic to the
5-dimensional non-solvable (non-Lie) Malcev algebra.

The paper of Kuzmin \cite{Kuzmin2} also classifies the 6-dimensional nilpotent
Malcev algebras.  It would be interesting to extend this to a classification of
all 6-dimensional Malcev algebras, and to study their enveloping algebras.

\section{Universal property of the octonions} \label{shestakovproof}

The goal of this research program is to study the universal nonassociative enveloping
algebra of the 7-dimensional simple Malcev algebra. Elduque \cite{Elduque} showed
that this Malcev algebra does not have a 6-dimensional subalgebra, so we have to
jump directly from the 5-dimensional non-solvable algebra to the 7-dimensional simple
algebra. The universal nonassociative enveloping algebra is linearly isomorphic to the
polynomial algebra in 7 variables, hence infinite dimensional.  But the universal
alternative enveloping algebra is the octonion algebra; the following proof was
explained to us by Shestakov.

\begin{theorem} \label{U(M)=O}
The universal alternative enveloping algebra $U(\mathbb{M})$ of the 7-di\-men\-sion\-al
simple Malcev algebra $\mathbb{M}$ over $\mathbb{R}$ is isomorphic to the division
algebra $\mathbb{O}$ of real octonions: $U(\mathbb{M}) \cong \mathbb{O}$.
\end{theorem}

\begin{proof}
The octonion algebra $\mathbb{O}$ has a basis $\{ 1, e_1, \hdots, e_7 \}$ for
which the multiplication is completely determined by the following conditions
for $1 \le i \ne j \le 7$ with the convention that $e_{i+7} = e_i$; see Kuzmin and
Shestakov \cite{KuzminShestakov}, page 217:
  \begin{equation} \label{star}
  e_i^2 = -1,
  \quad
  e_i e_j = - e_j e_i,
  \quad
  e_i e_{i+1} = e_{i+3},
  \quad
  e_{i+1} e_{i+3} = e_i,
  \quad
  e_{i+3} e_i = e_{i+1}.
  \end{equation}
In the Malcev algebra $\mathbb{O}^-$, the subalgebra of traceless (or pure imaginary)
octonions has basis $\{ e_1, \hdots, e_7 \}$; its multiplication is completely determined
by the following conditions:
  \begin{equation} \label{eq2}
  [ e_i, e_{i+1} ] = 2 e_{i+3},
  \qquad
  [ e_{i+1}, e_{i+3} ] = 2 e_i,
  \qquad
  [ e_{i+3}, e_i ] = 2 e_{i+1}.
  \end{equation}
The 7-dimensional simple Malcev algebra $\mathbb{M}$ is isomorphic to the Malcev
algebra of traceless octonions; see Kuzmin \cite{Kuzmin1}, Theorem 10. Therefore,
$\mathbb{M}$ has basis $\{ e_1, \hdots, e_7 \}$ with the structure constants \eqref{eq2}.

Assume that $\mathbb{M}$ is embedded in an alternative algebra $\mathbb{A}$ in
which multiplication is denoted by juxtaposition, and consider the injective Malcev
algebra homomorphism sending $e_i \in \mathbb{M} \to e_i \in \mathbb{A}^-$.
The following identity holds in $\mathbb{A}$ since it holds in every alternative
algebra as can be easily shown by direct expansion:
  \begin{equation} \label{associator}
  (x,y,z) = \tfrac16 J(x,y,z),
  \qquad
  J (x,y,z) = [[x,y],z] + [[y,z],x] + [[z,x],y].
  \end{equation}
Using \eqref{associator} and the structure constants \eqref{eq2} we obtain
  \begin{align}
  ( e_{i+1}, e_{i+2}, e_{i+5} ) &= \tfrac16 J( e_{i+1}, e_{i+2}, e_{i+5} ) = 2 e_i,
  \notag
  \\
 ( e_{i+4}, e_{i+1}, e_{i+6} ) &= \tfrac16 J( e_{i+4}, e_{i+1}, e_{i+6} ) = 2 e_i,
 \label{eq3}
  \\
  ( e_{i+1}, e_{i+3}, e_{i+2} ) &= \tfrac16 J( e_{i+1}, e_{i+3}, e_{i+2} ) = 2 e_{i+6}.
  \notag
  \end{align}
Every alternative algebra satisfies the polynomial identity
  \begin{equation} \label{identity1}
  [x,y] \circ (x,y,z) = 0,
  \end{equation}
where $\circ$ denotes the product $x \circ y = xy + yx$; see Zhevlakov et al.
\cite{Zhevlakov}, page 35 and Lemma 9, page 145.  From this identity, using
\eqref{eq2} and \eqref{eq3}, we obtain
  \begin{equation}
  e_i \circ e_{i+4} = 0,
  \qquad
  e_i \circ e_{i+2} = 0,
  \qquad
  e_i \circ e_{i+6} = 0.
  \end{equation}
Therefore (recalling that $e_{i+7} = e_i$) we have $e_i \circ e_j = 0$ for all $i \ne j$,
or equivalently
  \begin{equation} \label{eq5}
  e_i e_j = - e_j e_i.
  \end{equation}
From \eqref{eq2} and \eqref{eq5} it follows that
  \begin{equation} \label{starstar}
  e_i e_{i+1} = e_{i+3},
  \qquad
  e_{i+1} e_{i+3} = e_i,
  \qquad
  e_{i+3} e_i = e_{i+1}.
  \end{equation}
Therefore for all $i \ne j$ we have $e_i e_j = \pm e_k$ for some $k$.
Linearizing \eqref{identity1} gives
  \begin{equation} \label{identity1linear}
  [x,y] \circ (r,s,z) +
  [r,y] \circ (x,s,z) +
  [x,s] \circ (r,y,z) +
  [r,s] \circ (x,y,z) = 0.
  \end{equation}
We set $x = e_{i+2}$, $y = e_{i+6}$, $r = e_{i+1}$, $s = e_{i+2}$, $z = e_{i+5}$. On
the left side of \eqref{identity1linear} the two middle terms vanish, and we obtain
  \begin{equation} \label{short}
  [e_{i+2},e_{i+6}] \circ (e_{i+1},e_{i+2},e_{i+5})
  +
  [e_{i+1},e_{i+2}] \circ (e_{i+2},e_{i+6},e_{i+5})
  = 0.
  \end{equation}
Now \eqref{eq2}, \eqref{eq3} and \eqref{short} give
  \begin{equation}
  2 e_i \circ 2 e_i - 2 e_{i+4} \circ 2 e_{i+4} = 0,
  \quad \text{hence} \quad
  8 e_i^2 - 8 e_{i+4}^2 = 0.
  \end{equation}
Therefore $e_i^2 = e_{i+4}^2$ for all $i$, and hence $e_i^2 = e_j^2$ for all $i, j$;
we denote the common value by $a$.

Let $\mathbb{A}_0 \subseteq \mathbb{A}$ be the subalgebra generated by
$e_1, \hdots, e_7$; then $\mathbb{A}_0$ is also generated by $e_1, e_2, e_3$
since (8) gives
$e_1 e_2 = e_4$, $e_2 e_3 = e_5$, $e_3 e_4 = e_6$, $e_4 e_5 = e_7$.
Now consider the following subspace of $\mathbb{A}_0$:
  \begin{equation}
  \mathbb{A}'_0 = \mathbb{R} a + \sum_{i=1}^7 \mathbb{R} e_i. \nonumber
  \end{equation}
We claim that $\mathbb{A}'_0$ is a subalgebra of $\mathbb{A}_0$ and that
$-a$ is the identity element of this subalgebra.  The proof is in three parts.
\emph{Part 1:} We have already seen that $e_i^2 = a$ for all $i$. By (8) we have
$e_i e_{i+1} = e_{i+3}$. Multiplying this equation on the left by $e_i$ we obtain
$e_i^2 e_{i+1} = e_i e_{i+3}$. But equations (7) and (8) imply that
$e_i e_{i+3} = - e_{i+1}$, and so $a e_{i+1} = - e_{i+1}$. In a similar manner,
we have $e_{i+1} a = - e_{i+1}$.
\emph{Part 2:} For $a^2$, we use $e_i^2 = a$ for all $i$ and Part 1 to obtain
$a^2 = a e_i^2 = (a e_i)e_i = (-e_i) e_i = - e_i^2$. Hence $a^2 = -a$.
\emph{Part 3:} We have already observed that for all $i \ne j$ we have
$e_i e_j = \pm e_k$ for some $k$.
This completes the proof of the claim.

Since $\mathbb{A}'_0$ is also generated by $e_1, e_2, e_3$ we see that
$\mathbb{A}_0 = \mathbb{A}'_0$. It is now clear that $\mathbb{A}_0$ and $\mathbb{O}$
are isomorphic over $\mathbb{R}$ since they have the same structure constants.
\end{proof}

\begin{remark}
Theorem \ref{U(M)=O} also follows easily from Example 2 in Shestakov
\cite{Shestakov1}. In fact, in the notation of \cite{Shestakov1} we have
  \[
  8
  \le
  \dim U(\mathbb{M})
  =
  \dim \mathrm{gr} \, U(\mathbb{M})
  \le
  \dim \widetilde{S}(\mathbb{M})
  =
  8,
  \]
hence $\dim U(\mathbb{M}) = 8$ and $U(\mathbb{M}) \cong \mathbb{O}$. (We thank
the referee for pointing this out.)
\end{remark}

Shestakov and Zhelyabin \cite{ShestakovZhelyabin} determined the nucleus $N$
(elements which associate with all elements) and center $Z$ (elements which commute
and associate with all elements) for the universal nonassociative enveloping
algebras over a field $F$ of characteristic 0 of the 4-dimensional solvable
algebra $\mathbb{S}$, the 5-dimensional nilpotent algebra $\mathbb{T}$, the
5-dimensional non-solvable algebra $\mathbb{U}$, and the 7-dimensional simple
algebra $\mathbb{M}$:
  \allowdisplaybreaks
  \begin{alignat*}{4}
  N(U(\mathbb{S})) &= F[d],
  &\quad
  N(U(\mathbb{T})) &= F[c,e],
  &\quad
  N(U(\mathbb{U})) &= F,
  &\quad
  N(U(\mathbb{M})) &= F[c],
  \\
  Z(U(\mathbb{S})) &= F,
  &\quad
  Z(U(\mathbb{T})) &= F[e],
  &\quad
  Z(U(\mathbb{U})) &= F,
  &\quad
  Z(U(\mathbb{M})) &= F[c].
  \end{alignat*}
For $U(\mathbb{M})$, we set $c = e_1^2 + \cdots + e_7^2$, so that $N(U(\mathbb{M})) = Z(U(\mathbb{M})) = F[c]$, which is isomorphic to the polynomial ring in one variable.

\section{Bol algebras}

In an associative algebra, the commutator $[a,b]$ satisfies the Jacobi identity, and
the associator $(a,b,c)$ is identically zero.  In an alternative algebra, $[a,b]$ satisfies
the Malcev identity, and $(a,b,c)$ is an alternating function of its arguments; moreover,
the associator can be expressed in terms of the commutator:
  \[
  (a,b,c) = \tfrac16 \big( [[a,b],c] + [[b,c],a] + [[c,a],b] \big).
  \]
The next step is right alternative algebras, which satisfy only $(b,a,a) = 0$.
In this case we need both operations, commutator and associator. In fact, it is
more convenient to consider not the usual associator but rather the Jordan
associator:
  \[
  \langle a, b, c \rangle = ( a \circ b ) \circ c - a \circ ( b \circ c ),
  \qquad
  a \circ b = \tfrac12 ( ab + ba ).
  \]
In any right alternative algebra, $[a,b]$ and $[a,b,c] = \langle b, c, a
\rangle$ (note the cyclic shift) satisfy the following identities, which define
the variety of Bol algebras:
  \allowdisplaybreaks
  \begin{align*}
  &
  [a,a] = 0,
  \qquad
  [a,a,c] = 0,
  \qquad
  [a,b,c] + [b,c,a] + [c,a,b] = 0,
  \\
  &
  [a,b,[c,d,e]] - [[a,b,c],d,e] - [c,[a,b,d],e] - [c,d,[a,b,e]] = 0,
  \\
  &
  [[a,b],[c,d]] - [a,b,[c,d]] + [c,d,[a,b]] + [[a,b,c],d] - [[a,b,d],c] = 0.
  \end{align*}
That is, a Bol algebra is a Lie triple system with respect to the operation $[a,b,c]$,
and also has an anticommutative operation $[a,b]$; the two are related by the last
identity above.  Any Lie triple system becomes a Bol algebra by defining $[a,b] = 0$
for all $a, b$.  Universal nonassociative enveloping algebras for Bol algebras have
been obtained by P\'erez-Izquierdo \cite{PerezIzquierdo1}. The structure constants
have not been worked out explicitly even for the 2-dimensional simple Lie triple
system with basis $e, f$ and relations $[e,f,e] = 2 e$ and $[e,f,f] = -2 f$.

\section{Beyond Bol algebras}

To go further we need to recall the origins of Lie, Malcev and Bol algebras in
differential geometry.  A Lie algebra is the natural algebraic structure on the
tangent space at the identity element of a Lie group. Similarly, Malcev
(resp.~Bol) algebras arise as tangent algebras of analytic Moufang (resp.~Bol)
loops. A Lie algebra determines a local Lie group uniquely up to isomorphism,
and similarly a Malcev (resp.~Bol) algebra determines a local Moufang
(resp.~Bol) loop uniquely up to isomorphism. In general, one may consider any
analytic loop and the operations induced by its multiplicative structure on the
tangent space at the identity.  The resulting algebraic structure was described
from the geometric point of view by Mikheev and Sabinin \cite{MikheevSabinin},
and was clarified from the algebraic point of view by Shestakov and Umirbaev
\cite{ShestakovUmirbaev} and P\'erez-Izquierdo \cite{PerezIzquierdo2}.  These
algebraic structures are now called Sabinin algebras.

In the general case, a Sabinin algebra has an infinite sequence of operations;
from the algebraic point of view these operations correspond to the elements of
the free nonassociative algebra which are primitive in the sense of
(nonassociative) Hopf algebras.  Following this approach, we see that Lie,
Malcev and Bol algebras are the Sabinin algebras corresponding to associative,
alternative, and right alternative algebras.  Beyond right alternative
algebras, the next weaker form of associativity leads to third-power
associative algebras defined by the identity $(a,a,a) = 0$; these algebras
correspond to monoassociative loops, defined by the identity $a^2 a = a
a^2$. In order to determine a local monoassociative loop uniquely up to
isomorphism, the corresponding tangent algebra requires not only a binary
operation (the commutator) and a ternary operation (the associator), but also a
quaternary operation, called a quaternator by Akivis and Goldberg
\cite{AkivisGoldberg}.

These binary-ternary-quaternary structures have been studied from the geometric
point of view by Akivis and Shelekhov \cite{AkivisShelekhov}, Shelekhov
\cite{Shelekhov}, and Mikheev \cite{Mikheev}. We will say a few words about
them from an algebraic point of view. Let $A$ be the free nonassociative
algebra on a set $X$ of generators. Define a coproduct $\Delta\colon A \to A
\otimes A$ by setting $\Delta(x) = x \otimes 1 + 1 \otimes x$ for $x \in X$ and
extending to all of $A$ by assuming that $\Delta$ is an algebra morphism:
$\Delta(fg) = \Delta(f) \Delta(g)$. The multilinear operations on $A$
correspond to the polynomials $f$ which are primitive in the Hopf algebra
sense: $\Delta(f) = f \otimes 1 + 1 \otimes f$. Shestakov and Umirbaev
\cite{ShestakovUmirbaev} gave a complete set of these primitive elements.  In
degrees 2 and 3 there are only the commutator and associator, but in degree 4
there are two quaternary operations,
  \allowdisplaybreaks
  \begin{align*}
  f &= \langle a, b, c, d \rangle = (ab,c,d) - a(b,c,d) - (a,c,d)b,
  \\
  g &= \{ a, b, c, d \} = (a,bc,d) - b(a,c,d) - (a,b,d)c,
  \end{align*}
which measure the deviation of the associators $(-,c,d)$ and $(a,-,d)$ from
being derivations. In degree 4, we also have six elements (called Akivis elements)
that are generated by the operations of lower degree (the commutator and associator):
  \[
  [[[a,b],c],d], \;\;
  [(a,b,c),d], \;\;
  [[a,b],[c,d]], \;\;
  ([a,b],c,d), \;\;
  (a,[b,c],d), \;\;
  (a,b,[c,d]).
  \]
Every primitive multilinear nonassociative polynomial of degree 4 is a linear
combination of permutations of these six Akivis elements and the two non-Akivis
elements $f$ and $g$. At this point we have two quaternary operations, $f$ and
$g$, but we have not yet imposed third-power associativity.  If we assume
$(a,a,a) = 0$ then $g$ is a consequence of $f$, in the sense that $g$ is a linear
combination of permutations of $f$, the Akivis elements, and the consequences
in degree 4 of the linearization of the third-power associative identity $(a,a,a) = 0$:
  \begin{align*}
  ([a,d],b,c) + ([a,d],c,b) + (b,[a,d],c) + (b,c,[a,d]) + (c,[a,d],b) + (c,b,[a,d]) &= 0,
  \\
  [(a,b,c),d] + [(a,c,b),d] + [(b,a,c),d] + [(b,c,a),d] + [(c,a,b),d] + [(c,b,a),d] &=
  0.
  \end{align*}
The next step is to find the polynomial identities relating these three
operations: $[a,b]$, $(a,b,c)$, $\langle a,b,c,d \rangle$. These identities
will define a variety of nonassociative multioperator algebras generalizing Bol
algebras.

\section*{Acknowledgements}

We thank the organizers of the Second Mile High Conference on Nonassociative
Mathematics for a very stimulating meeting. We are especially grateful to Ivan
Shestakov for the proof that $U(\mathbb{M}) \cong \mathbb{O}$.  Bremner,
Tvalavadze and Usefi were partially supported by NSERC. Peresi was partially
supported by CNPq.

\end{document}